\def\draft{n}
\documentclass{amsart}
\usepackage[headings]{fullpage}
\usepackage{amssymb,epic,eepic,epsfig}
\usepackage{pb-diagram}

\theoremstyle{plain}

\newtheorem{theorem}{Theorem}
\newtheorem*{lawrence}{Lawrence's Theorem}
\newtheorem{proposition}{Proposition}[section]
\newtheorem{lemma}[proposition]{Lemma}
\newtheorem{corollary}[proposition]{Corollary}

\theoremstyle{definition}
\newtheorem{definition}[proposition]{Definition}

\theoremstyle{remark}
\newtheorem{example}[proposition]{Example}

\def\printname#1{
        \if\draft y
                \smash{\makebox[0pt]{\hspace{-0.5in}
                        \raisebox{8pt}{\tt\tiny #1}}}
        \fi
}

\newlength{\standardunitlength}
\catcode`\@=11
\long\def\@makecaption#1#2{%
     \vskip 10pt

\setbox\@tempboxa\hbox{
       \small\sf{\bfcaptionfont #1. }\ignorespaces #2}%
     \ifdim \wd\@tempboxa >\captionwidth {%
         \rightskip=\@captionmargin\leftskip=\@captionmargin
         \unhbox\@tempboxa\par}%
       \else
         \hbox to\hsize{\hfil\box\@tempboxa\hfil}%
     \fi}
\font\bfcaptionfont=cmssbx10 scaled \magstephalf
\newdimen\@captionmargin\@captionmargin=2\parindent
\newdimen\captionwidth\captionwidth=\hsize
\catcode`\@=12

\def\lbl#1{\label{#1}\printname{#1}}

\newcommand{\Q}{\mathbb Q}
\newcommand{\Z}{\mathbb Z}

\newcommand{\Psibar}{\overline{\Psi}}

\newcommand{\ba}{\begin{align*}}
\newcommand{\ea}{\end{align*}}
\newcommand{\ip}[2]{\langle #1,#2\rangle}

\def\BZ{\mathbb Z}

\def\BQ{\mathbb Q}

\def\BC{\mathbb C}

\def\P{\mathcal P}

\def\La{\Lambda}

\def\faceof{\prec}

\def\la{\langle}
\def\ra{\rangle}

\def\Cone{\mathrm{Cone}}
\def\cone{\mathrm{cone}}

\def\calL{\mathcal{L}}

\def\calC{\mathcal{C}}
\def\calM{\mathcal{M}}

\def\lin{\mathrm{lin}}

\def\Supp{\mathrm{Supp}}

\def\lin{\mathrm{lin}}
\def\Tan{\mathrm{Tan}}

\def\La{\Lambda}
\def\lin{\mathrm{lin}}
\def\aff{\mathrm{aff}}

\def\calC{\mathcal{C}}
\def\itilde{\tilde{I}}
\def\stilde{\tilde{S}}
\def\mutilde{\tilde{\mu}}

\def\r{r^\Psi} 

\def\rsigma{r^\Psi(\sigma)}

\def\Td{{\rm Td}}

\def\Div{{\rm Div}}
\def\Hom{{\rm Hom}}
\def\Cone{{\rm Cone}}

\begin{document}


\title[Cycle-level products in equivariant cohomology of toric varieties]{Cycle-level products in equivariant cohomology of toric varieties}
\author{Benjamin P. Fischer}
\address{Department of Mathematics and Statistics\\
Boston University\\
111 Cummington Street\\
Boston, MA 02215}
\email{benfischeris@gmail.com}
\author{James E. Pommersheim}
\address{Mathematics Department \\
         Reed College \\
         3203 SE Woodstock Boulevard \\
Portland, Oregon 97202-8199}
\email{jamie@reed.edu}


\date\today


\begin{abstract}

In this paper, we define an action of the group of equivariant Cartier divisors on a toric variety $X$ on the equivariant cycle groups of $X$, arising naturally from a choice of complement map on the underlying lattice. If $X$ is nonsingular, this gives a lifting of the multiplication in equivariant cohomology to the level of equivariant cycles.  As a consequence, one naturally obtains an equivariant cycle representative of the equivariant Todd class of any toric variety.  These results extend to equivariant cohomolgy the results of \cite{Th} and \cite{PT}.  In the case of a complement map arising from an inner product, we show that the equivariant cycle Todd class obtained from our construction is identical to the result of the inductive, combinatorial construction of Berline-Vergne \cite{BV1, BV2}.  In the case of arbitrary complement maps, we show that our Todd class formula yields the local Euler-Maclarurin formula introduced in \cite{GP}.

\end{abstract}

\maketitle

\tableofcontents

\section{Introduction}
\lbl{sec.intro}

\subsection{Overview}
\lbl{sub.overview}

Intersection theory on a nonsingular algebraic variety provides a natural intersection product of cycles modulo rational equivalence.  One might wonder in what circumstances there is a reasonable lifting of this product to the level of cycles, so that any two cycles on $X$ on can be multiplied to produce a well-defined cycle on $X$ in a natural way that respects rational equivalence.  If the cycles intersect properly, there is a natural product, but if the intersection is not proper, then one must settle for knowing the product only as a cycle modulo rational equivalence. More generally, for arbitrary (possibly singular) algebraic varieties, intersection theory provides an action of the Picard group of an arbitrary algebraic variety on the Chow groups of the variety. One can ask if there is a natural lifting of the action of the Picard group to the level of algebraic cycles.  For toric varieties, such an action was constructed in \cite{Th}.  The action depends on the choice 
 of a complement map, which is a certain global choice of linear subspaces.  (See Section \ref{sub.statementofresults} for details.)

One of the motivations for a cycle-level intersection theory is that it leads naturally to a cycle expression for the Todd class of a toric variety.   Indeed, the Todd class of a nonsingular toric variety has a well-known expression as a product of torus-invariant cycles, so given the cycle-level multiplication derived from a complement map, one obtains a natural expression for the Todd class of a toric variety.  This is worked out in \cite{PT}, where one finds a cycle expression for the Todd class of an arbitrary toric variety with rational coefficents depending only on the local information in the fan, giving an answer to a question of Danilov. From this Todd class formula, via a well-known application of Riemann-Roch, one obtains a local formula for the number of lattice points in a integral polytope. 

The purpose of this paper is to extend the results of \cite{Th} and \cite{PT} to equivariant cohomology.  In particular, given a complement map, we produce a natural action of the equivariant divisor group on the equivariant cycle group.  In the simplicial case, we obtain a natural ring structure on the group of equivariant cycles tensored with $\Q$. As a consequence of these results, for any choice of complement map, we obtain a natural local and computable expression for the equivariant Todd class of an arbitrary toric variety.   We show that the expression for the equivariant Todd class so obtained, in the special case of a complement map arising from an inner product, is equivalent to that obtained by \cite{BV2} by an entirely different, purely combinatorial recipe.   

Finally, we relate our expression of the Todd class to the Euler-Maclaurin formulas of \cite{GP}.  In that paper, it is shown that a complement map naturally gives rise to a function $\mu$ on cones that interpolates between exponential sums and integrals, and hence gives rise to a local Euler-Maclaurin formula.   In this paper, we show that the functions $\mu$ arising from our Todd class construction are identical to the functions $\mu$ constructed in \cite{GP} through a different inductive combinatorial method.   As a corollary, we prove a conjecture of \cite{GP} which asserts that the constant term of the power series constructed in \cite{BV2} and \cite{GP} agree with the Todd class coefficients constructed in \cite{PT}.

\subsection{Definition of the action and basic properties}
\lbl{sub.statementofresults}

We now give the details of our construction.  Let $X = X_{\Sigma}$ be a toric variety defined by a fan $\Sigma$ in a lattice $N$, with associated torus $T$.  We follow notation that is standard in the theory of toric varieties (cf. \cite{Fu}.)  Let $M=\Hom(N,\BZ)$ denote the lattice dual to $N$, and let $\Lambda=\Z[M]$.  For each cone $\sigma\in\Sigma$, let $N_{\sigma}$ denote the lattice $L_{\sigma} \cap N$, where $L_{\sigma}$ is the linear span of $\sigma$, and let $N(\sigma)= N/N_{\sigma}$.  Dual to $N(\sigma)$ and $N_{\sigma}$ are the lattices $M(\sigma)= M\cap \sigma^\perp$ and $M_{\sigma} = M/M(\sigma)$.   We will use $\ip{\cdot}{\cdot}$ to denote the natural pairing $M\times N\rightarrow \BZ$ or indeed $M(\sigma)\times N(\sigma)\rightarrow \BZ$ for any $\sigma\in\Sigma$.

For any abelian group $L$ (such as $M$, $N$, $\Lambda$, etc.), we denote $L_\Q := L \otimes \Q$. 

If a cone $\tau\in\Sigma$ contains $\sigma$ as a maximal proper face, we will write $\tau\rightarrow\sigma$.  In this case, the image of $\tau$ in $N(\sigma)$ is a 1-dimensional cone, and we use $n_{\tau,\sigma}$ to denote the unique primitive element of $N(\sigma)$ contained in this cone. 

Recall that for each $\sigma\in\Sigma$ there is a $T$-invariant subvariety $V(\sigma)\subset X$.  

We wish to define a cycle-level action of the group $\Div_T(X)$ of equivariant $\Q$-Cartier divisors on the equivariant cycle groups $Z^T_*(X)$.
 
\begin{definition}  Let $X$ be the toric variety associated to a fan $\Sigma$.  The  {\it cycle group} of $X$, denoted $Z_*(X)$, is the free abelian group generated by $ \{ V(\sigma) \ |\  \sigma \in\Sigma\}$.  The {\it equivariant cycle group (with rational coefficients)}  $Z^T_*(X)$ is defined as $Z^T_*(X) := Z_*(X) \otimes \Lambda_\Q$. 

\end{definition}

Since the classes of invariant cycles generate the Chow groups of $X$, there is a natural surjection $Z_*(X)\rightarrow A_*(X)$.
Similarly, there is a natural surjection $Z_*^T(X)\rightarrow A_*^T(X)_\Q$; indeed, \cite{Br2} gives a presentation of the equivariant Chow groups $A_*^T(X)$ as a quotient of $Z_*^T(X)$. For $\sigma\in\Sigma$, we use $V_{\sigma}$ to denote the cycle $V(\sigma)$ considered as an element of $Z^T_*(X)$.

Denote by $\Div(X)$ the group of $T$-invariant $\Q$-Cartier divisors.  Recall that a $T$-invariant $\Q$-Cartier divisor on $X$ is given by local equations $\{d_{\sigma}\}_{\sigma\in\Sigma}$, where each $d_{\sigma}\in M_{\sigma, \Q}$.  These local equations are compatible in the sense that if $\sigma,\tau\in\Sigma$ with $\sigma\subset\tau$, then $d_{\sigma}$ is the image of $d_{\tau}$ under the natural map $M_{\tau, \Q} \rightarrow M_{\sigma, \Q}$.   Let $\Div_T(X)=\Div(X)\otimes\Lambda_\Q$. This is the $\Lambda_\Q$-module that acts on $Z_*^T(X)_\Q$, as we assert below.

The idea of a {\it complement map} was introduced in \cite{Th}.  We give here a quick definition which is easily seen to be equivalent to the notion of rigid complement map used in \cite{Th}.

\begin{definition}
Let $\calL$ be a set of $\Q$-subspaces of $N_\Q$ which contains $N_\Q$.  Then a {\it complement map} $\Psi$ assigns, to each $L_1,L_2 \in \calL$ such that $L_1 \subset L_2$, a section $i^\Psi: L_1^* \rightarrow L_2^*$ of the natural (restriction) map $L_2^* \rightarrow L_1^*$.  The sections are assumed to be transitive with respect to inclusion.  Given a fan $\Sigma$ in $N$, we say that $\Sigma$ is {\it $\Psi$-generic} if $\{N_{\sigma, \Q} \}_{\sigma\in\Sigma}$ is a subset of $\calL$. 
\end{definition}

In particular, if $\Sigma$ is $\Psi$-generic, then we obtain a section $i^\Psi: M_{\sigma, \Q} \rightarrow M_\Q$ for each $\sigma \in \Sigma$.  If $d\in M_{\sigma, \Q}$, denote its image $i^{\Psi}(d) \in M_\Q$ by $d^{\Psi}$.   Denote the $\Q$-subspace of $M_\Q$ generated by $i^\Psi(\sigma)$ as $\Psi(\sigma)$.  One easily checks that $\Psi(\sigma) \oplus M(\sigma)_\Q = M_\Q$.  In this sense, we may say that a complement map is a choice of complementary spaces.

The following theorem asserts that a complement map gives rise to a natural action of $\Div_T(X)$ on $Z^T_*(X)$.   In the formula that defines this action, it is useful to note that if $D$ is a Cartier divisor with local equations $\{d_{\sigma}\}_{\sigma\in\Sigma}$ and $\tau\rightarrow\sigma$, then $d_\tau$ and $d_\sigma^\Psi$ have the same image in $M_{\sigma, \Q}$.   Thus we can consider the difference $d_\tau-d_\sigma^\Psi$ to be an element of $M(\sigma)_\Q$.  As such the expression $\ip{d_\tau-d_\sigma^\Psi}{n_{\tau,\sigma}}$ is well-defined.

\begin{theorem}
\lbl{thm.action}

Let $X=X(\Sigma)$ be the toric variety associated to a fan $\Sigma$ in $N$. Given a generic complement map $\Psi$ on $N$, there is a natural action of the group $\Div_T(X)$ on $Z^T_*(X)$, given as follows:  for $D\in \Div(X)$ with local equations $\{d_{\sigma}\}_{\sigma\in\Sigma}$:

\begin{equation}
\lbl{eq.action}
D \cdot V_{\sigma}  = \sum_{\tau:\tau \rightarrow \sigma} \ip{d_\tau-d_\sigma^\Psi}{n_{\tau,\sigma}}  V_{\tau} + d_\sigma^\Psi V_{\sigma}.
\end{equation}

The action satisfies the following properties:

\begin{itemize}

\item (Lifting) The action defined above is a lifting of the action of the Picard group of $X$ on the equivariant Chow groups $A_*^T(X)_\Q$.

\item (Commutativity) For any $D, E \in \Div_T(X)$ and $C \in Z_*^T(X)_\Q$,  we have $D \cdot (E \cdot C) = E \cdot (D \cdot C)$.

\item (Compatibility with non-equivariant cycle-level intersection)  The action defined above is a lifting of the cycle-level action of $\Div(X)$ on $Z(X)_\Q$ defined in \cite{Th}.

\end{itemize}  

\end{theorem}

\subsection{Ring structure in the simplicial case}
\lbl{sub.simplicialresults}

In the case when $X$ is a simplicial toric variety, the action defined above provides a ring structure on $Z^T_*(X)$.  As we show in Section \ref{sec.simplicial}, this is a consequence of Theorem \ref{thm.action} together with the well-known fact that every Weil divisor on $X$ is represented by a Cartier divisor.  The ring structure obtained is a lifting of the product structure on the equivariant cohomology ring $A^*_T(X)_\Q$, which is described (over $\Z$) in \cite{Fu2}. We now describe the ring stucture on $Z^T_*(X)$ explicitly.

\begin{theorem}
\lbl{thm.ringstructure}

Let $\Sigma$ be a simplicial fan and let $X$ be the associated toric variety.  Let $\Psi$ be a generic complement map. Then the action of Theorem \ref{thm.action} induces a commutative ring structure on $Z^T_*(X)$.  This product is a lifting of the product on the equivariant cohomology ring $A^*_T(X)_{\BQ}$ under the natural surjection $Z_*^T(X)_{\BQ}\rightarrow A^*_T(X)_{\BQ}$.

Furthermore, the ring structure on $Z^T_*(X)$ is described explicitly as follows.  Let $\Sigma_{(1)} = \{\rho_1, \dots, \rho_s\}$ be the $1$-dimensional cones of $\Sigma$, and let $D_i:=V_{\rho_i}\in Z^T_*(X)$ be the corresponding cycles.  Denote by $n_i$ the primitive element of $N \cap\rho_i$.  Then

$$Z^T_*(X) \cong \frac{\Lambda_\BQ[D_1,\dots, D_s]}{I + J^\Psi}$$

where $I$ is the Stanley-Reisner ideal

$$I = \langle D_{i_1} D_{i_2} \dots D_{i_k} : \rho_{i_1} + \rho_{i_2}+ \dots + \rho_{i_k} \notin \Sigma \rangle $$

and

$$J^\Psi = \biggl\langle D_{i_1} D_{i_2} \dots D_{i_k}\biggr(\sum_{j=1}^s \ip{m}{n_j} D_j - m\biggl) : \rho_{i_1} + \rho_{i_2} + \dots + \rho_{i_k} = \sigma \in \Sigma, m \in \Psi( \sigma) \biggr\rangle . $$

\end{theorem}

If we modify the definition of $J^\Psi$ above by allowing all $m\in M_\Q$ (eliminating the restriction $m \in \Psi( \sigma)$), we create a larger ideal $J'$.  The equivariant cohomology ring over $\Q$, denoted $A_T^*(X)_\Q$, then has a natural presentation as ${\Lambda_\Q[D_1,\dots, D_s]}/(I + J')$, as explained in \cite{Fu2}.  Note the equivariant cohomology ring can be defined over $\Z$ in the nonsingular case.  However, the ring structure in Theorem \ref{thm.ringstructure} cannot, since our complement maps are in general only well-defined over $\Q$.

\subsection{A Cycle Equivariant Todd class }
\lbl{sub.toddresults}

We now show how Theorem \ref{thm.ringstructure} can be used to obtain a cycle expression for the equivariant Todd class of an arbitrary toric variety, given the choice of a complement map.  Recall that for a nonsingular toric variety $X$, the equivariant Todd class, denoted $\Td^T(X)$, which naturally lives in the equivariant cohomology ring $A_T^*(X)_{\BQ}$, can be expressed as

\begin{equation}
\lbl{eq.toddprod}
\Td^T(X) := \prod_{i=1}^s  \frac{D_i}{1-e^{-D_i}} .
\end{equation}
(See \cite{BV3} for a proof.)  Using Theorem 2, we can multiply out the above product, and the result is a cycle expression for the equivariant Todd class, living in the completion $\hat{Z}^T_*(X)_{\BQ}$ of  $Z^T_*(X)$.  This gives us an expression for $\Td^T(X)$ in terms of the cycles $V_{\sigma}$:
$$
\Td^T(X) = \sum_{\sigma\in\Sigma} r^{\Psi}(\sigma) V_{\sigma},
$$
where the coefficients $r^{\Psi}(\sigma)$ live in  $\hat{\Lambda}_{\BQ} := \widehat{\Q[M]}$, the completion of $\Lambda_\Q$.  As we show, these coefficients depend only on the complement map $\Psi$ and on the nonsingular cone $\sigma$, and are independent of the rest of $\Sigma$.  Every cone $\sigma$ has a subdivision into nonsingular cones, allowing $\sigma$ to be written as the union of a finite collection of cones $\{\sigma_1,\dots\sigma_t\}$ all of the same dimension, which intersect only along boundaries.  In this case, we define

\begin{equation}
\lbl{eq.additivity}
r^{\Psi}(\sigma) = \sum_i r^{\Psi}(\sigma_i).
\end{equation}

It turns out, as asserted by our next theorem, that this sum is independent of the chosen subdivision, and thus we obtain a local formula for the equivariant Todd class of a toric variety.

\begin{theorem} 
\lbl{thm.rsigma}
Let $\Psi$ be a complement map on a lattice $N$.  The above construction produces a well-defined map
$$
r^{\Psi} : \{ \Psi\textrm{-generic\ cones\ in\ }N \} \longrightarrow \hat{\Lambda}_{\BQ}.
$$
This function satisfies the following properties
\begin{itemize}

\item (Local expression for $\Td^T(X)$)  For every $\Psi$-generic fan $\Sigma$, we have
\begin{equation}
\lbl{eq.cycletodd}
\Td^T(X) = \sum_{\sigma\in\Sigma} r^{\Psi}(\sigma) V_{\sigma},
\end{equation}

\item (Additivity)  If $\sigma=\cup_{i=1}^t \sigma_i $ is a subdivision of a $\Psi$-generic cone $\sigma$ in $N$ into cones $\sigma_i$, all of the same dimension, which intersect only along boundaries, then 
$$
r^{\Psi}(\sigma) = \sum_i r^{\Psi}(\sigma_i).
$$

\item (Analytic properties)
For any $\Psi$-generic cone $\sigma$, the power series $r^{\Psi}(\sigma)\in \hat{\Lambda}_{\BQ}$ represents a meromorphic function on $N\otimes\BC$, regular at the origin.

\item (Agreement with nonequivariant cycle Todd class)  The constant term $r^{\Psi}(\sigma)(0)$ agrees with cycle Todd class coefficients $\mu(\sigma)$ constructed in \cite{PT}.

\end{itemize}

\end{theorem}

\subsection{SI-Interpolators and Euler-Maclaurin formulas for polyhedra }
\lbl{sub.interpolator}

Finally, we indicate some connections with sum-integral interpolation and Euler-Maclaurin formulas for polyhedra.  A cycle expression for the Todd class allows one, via Riemann-Roch, to express exponential sums (S) over a polyhedron $P$ in terms of exponential integrals (I) over the faces $F$ of $P$.  

To make this connection explicit, suppose that $P$ is a rational polyhedron in $M$.  
One can associate to $P$ two meromorphic functions, the {\em exponential sum}
 $S(P) \in \calM(N)$ and the {\em exponential integral} $I(P) \in \calM(N)$
where $\calM(N)$ is the algebra of 
meromorphic functions on $N\otimes \BC$.  These functions are given by the equations
\begin{equation}
\lbl{eq.SandI}
S(P)(\xi)=\sum_{x \in P \cap M}  e^{\ip{\xi}{x}}, 
\qquad
I(P)(\xi)=\int_{P}  e^{\ip{\xi}{x}} dx
\end{equation}
for $\xi\in N\otimes \BC$ provided $|e^{\ip{\xi}{x}}|$ is summable (resp. 
integrable) over $P$.  
The fact that the equations define meromorphic functions, 
as well as the precise characterization and properties of the functions 
$S$ and $I$, is essentially  the content of Lawrence's theorem \cite{La}. 

The fact that the functions $r^{\Psi}$ defined above satisfy equation (\ref{eq.cycletodd}) implies that they interpolate between exponential sums and integrals, as stated in the following corollary.  If $F$ is a face of a lattice polyhedron $P$ in $M$, we use $\sigma_{P,F}$ to denote the cone in $N$ dual to the tangent cone to $P$ along $F$;  the cone $\sigma_{P,F}$ is the cone in the inner normal fan $\Sigma$ corresponding to $F$. (See Section \ref{sub.todd} for detailed definitions.)
\begin{corollary}
\lbl{thm.rinterpolates}
Let $\Psi$ be a complement map on $N$.  Then for any $\Psi$-generic integral polyhedron $P$ in $M$ we have
$$
S(P)(\xi)= \sum_F r^{\Psi}(\sigma_{P,F})(-\xi)\cdot I(F)(\xi),
$$
where the sum is taken over all faces $F$ of $P$.
\end{corollary}

Finally, we note that these results connect with existing results in \cite{PT,BV1,GP} and strengthen the connections between those results.   In the theorem below, we show that  our interpolators $r^{\Psi}$ agree with those constructed in \cite{GP} for general $\Psi$, and for $\Psi$ arising from an inner product, our $r^{\Psi}$ coincides with the $\mu$ constructed in \cite{BV1, BV2}.  Note, however, that both \cite{BV1} and \cite{GP} allow rational polyhedra, while in the current case we are restricted to integral polyhedra.

\begin{theorem}
\lbl{thm.gpagree}
Let $\Psi$ be a complement map on a lattice $N$, and let $\mu^{\Psi}$ be the SI-interpolator defined in \cite{GP}.  Let $\sigma$ be a cone in $N$ and  $\check{\sigma}$ be its dual in $M$.  We then have 
$$
r^{\Psi}(\sigma)(\xi)=\mu^{\Psi}(\check{\sigma})(-\xi).
$$
\end{theorem}

Note that for $\Psi$ derived from an inner product, the above theorem relates our  $r^{\Psi}$ to the $\mu$ defined in \cite{BV2} restricted to lattice cones.

As a corollary of the above theorem, using the last part of Theorem \ref{thm.rsigma}, we see that the constant term of the \cite{GP} construction coincides with the rational numbers $\mu$ from the \cite{PT} construction.  In particular, the constant term of the \cite{BV1} construction is the $\mu$ from \cite{PT} in the case of complement maps arising from inner products. This gives a positive answer to Conjecture 1 in \cite{GP}.

\begin{corollary}
Suppose $\Psi$ is a complement map on a lattice $N$. Let $\mu_0^{\Psi}$ denote the cycle Todd class coefficients constructed in \cite[Corollary 1]{PT}, and let $\mu^{\Psi}$ be the SI-interpolator defined in \cite{GP} .  Let $\sigma$ be a cone in $N$ and  $\check{\sigma}$ be its dual in $M$.  We then have
$$
\mu_0^{\Psi}(\sigma) =  \mu^{\Psi}(\check{\sigma})(0).
$$
\end{corollary}

\vskip0.25in
\noindent{\bf Acknowledgement.} We would like to thank William Fulton for useful conversations and for the inspiration he has provided over the years.

\vskip0.25in
\section{Details and Proofs}
\subsection{Constructing the Group Action}
\lbl{sec.action}

The following section is included in order to provide motivation for the unusual formula for the group action given in Theorem $\ref{thm.action}$.  In the process, we will prove why the action is a lift of the action of the Picard group on the equivariant Chow groups.  We also explain why the choice of a complement map appears to be necessary.

We require an action of $\Div_T(X)$ on $Z_*^T(X)$ that respects $\Lambda_\Q$-multiplication.  Since, as $\Lambda_\Q$-modules, $\Div_T(X)$ is generated by $\Div(X)$, and $Z_*^T(X)$ is generated by $\{V_\sigma: \sigma \in \Sigma\}$, it is enough to describe how $\Div(X)$ acts on each $V_\sigma$.  We first recall the natural map from divisors to cycles.  Let $D \in \Div(X)$ be a divisor with local equations $d_\sigma$.  Let $\rho_1, \dots, \rho_s $ be the one-dimensional cones in $\Sigma $.  For simplicity, we notate $V_i := V_{\rho_i}$.  Associated to $D$, we define the cycle

$$[D] := \sum_{i=1}^s \ip{d_{\rho_i}}{n_i} V_i,$$
where $n_i$ is the primitive generator of $\rho_i$ (the first lattice point along $\rho_i$).  The map $D \mapsto [D]$ clearly induces a $\Lambda$-module homomorphism from $\Div_T(X)$ to $Z_*^T(X)$. 

Let $\sigma \in \Sigma$.  We wish to consider the action of $D$ on $V_\sigma$ as an intersection product of $[D]$ and $V_\sigma$.  Formally, we consider

$$[D] \cdot V_\sigma = \sum_{i=1}^s \ip{d_{\rho_i}}{n_i} V_i \cdot V_\sigma$$

Now we break up the summation and use the mapping property of local equations:

$$[D] \cdot V_\sigma = \sum_{\rho_i \not \prec \sigma} \ip{d_{\rho_i}}{n_i} V_i \cdot V_\sigma + \sum_{\rho_i \prec \sigma} \ip{d_{\sigma}}{n_i} V_i \cdot V_\sigma$$

We have separated the terms so that the left-hand summation contains only terms that intersect properly, so that their intersection product is natural and well-defined.  However, the right-hand summation contains terms that do not intersect properly.  We recall that the intersection product is always well-defined in the equivariant Chow ring $A_*^T(X)_\Q$, so we consider $[D]$ as an element of this ring.  As described in \cite{Fu2} (over $\Z$), $A_*^T(X)_\Q \cong \Lambda_\Q[V_1,\dots,V_s] / J'$, where

$$J' = \biggl\langle \sum_{\tau: \tau \rightarrow \sigma} \ip{m}{n_{\tau,\sigma}}V_i - mV_\sigma : \sigma \in \Sigma, m \in M_\Q \biggr\rangle $$

The solution is to shift $[D]$ by an element of $J'$ in order to eliminate the terms that do not intersect properly with $\sigma$, thus preserving the intersection product in the Chow ring.  We would like to pick an element $E \in J'$:

$$E = \sum_{i=1}^s \ip{m}{n_i} V_i - mV_{\{0\}}$$

where $m \mapsto d_\sigma$ under the quotient map $M_\Q \rightarrow M_{\sigma, \Q}$.  It is clear that $[D] - E$ would then only contain terms that intersect properly with $D_\sigma$.  However, we run into some ambiguity: there is no natural choice of $m$, since there is no natural section $M_{\sigma, \Q} \rightarrow M_\Q$.  This is why we must choose a complement map.  

Let $\Psi$ be a complement map on $V$ with each cone in $\Sigma$ included in the domain.  Recall that we write $d_\sigma^\Psi$ as the image of $d_\sigma$ in $M_\Q$, under $\Psi$.  Define the cycle

$$E_\sigma^\Psi := \sum_{i=1}^s \ip{d_\sigma^\Psi}{n_i} V_i - d_\sigma^\Psi V_{\{0\}}$$

Note that $E_\sigma^\Psi$ is in $J'$ as required.  Thus the product $([D] - [E_\sigma^\Psi])\cdot V_\sigma$ is a product of cycles that intersect properly, and furthermore, each term in the product corresponds either to $\sigma$ or to a cone $\tau$ that contains $\sigma$ as a maximal proper face.  Then we define

$$D \cdot V_\sigma = ([D]-[E_\sigma^\Psi])\cdot V_\sigma$$.

When we work this out explicitly, we recover the formula in Theorem $\ref{thm.action}$: 

\begin{equation}
D \cdot V_\sigma = \sum_{\tau: \tau \rightarrow \sigma} \ip{d_\tau - d_\sigma^\Psi}{n_{\tau,\sigma}} V_\tau + d_\sigma^\Psi V_\sigma.
\end{equation}

\subsection{Welldefinedness and Properties}
\lbl{sec.thm1proof}

We now begin the proof of Theorem 1.
\begin{proof}

The fact that the map respects the properties of a group action follows immediately from the additivity of all maps involved in the definition.  The characterization of the map in the previous section proves that the action on $Z^T_*(X)$ is a lifting of the action on $A_*^T(X)_\Q$.

In order to prove that the action generalizes the cycle-level action defined in \cite{Th}, we must reconcile our definition with his.  Let $D \in \Div^T(X), C \in Z^T_*(X)$.  Let $\odot$ be the action of $\Div(X)$ on $Z_*(X)_\Q$ as defined in \cite{Th}.  Finally, let $\phi$ be the natural homomorphism from equivariant divisors (or cycles) to basic divisors (or cycles) that maps $M_\Q$ to $0$.  Then we wish to prove that

$$\phi(D \cdot C) = \phi(D) \odot \phi(C)$$.

Since $\phi$ is a module homomorphism, it suffices to assume $C = V_\sigma$ and $D \in \Div(X)$.  Let $\{d_\sigma\}$ be the local equations for $\phi(D)$.  For this proof, we briefly introduce notation from \cite{Th}.  As defined in \cite{Th},

$$\phi(D) \odot V_\sigma = \sum_{\tau: \tau \rightarrow \sigma} \ip{\pi_\sigma(m_\tau)}{n_{\tau,\sigma}} V_\tau$$

where $\pi_\sigma: M_\tau \rightarrow M(\sigma)_\tau$ is the projection map derived from $\Psi$.  Using the above definition,

$$\phi(D \cdot V_\sigma) = \sum_{\tau: \tau \rightarrow \sigma} \ip{m_\tau - m_\sigma^\Psi}{n_{\tau,\sigma}} V_\tau$$.

It suffices to show that for $\tau \rightarrow \sigma$,

$$\ip{m_\tau - m_\sigma^\Psi}{n_{\tau,\sigma}} = \ip{\pi_\sigma(m_\tau)}{n_{\tau,\sigma}}$$

Following through the definitions of each term makes the statement clear.  For purely aesthetic reasons, the authors have chosen to define the action using the embedding map $i^\Psi$ rather than the projection map $\pi_\sigma$, although they are certainly equivalent.

Since the basic cycle-level action was proven to be commutative in \cite{Th}, most of the work in proving the commutativity of the equivariant-level action is done.  Let $D, E \in \Div^T(X)$, $C \in Z^T_*(X)$.  We wish to show that $D \cdot (E \cdot C) = E \cdot (D \cdot C)$.  Once again, it suffices to assume $D, E \in \Div(X)$, $C = V_\sigma$.  Expanding upon the definition, we see that

$$D \cdot (E \cdot V_\sigma) = \sum_{\delta: \delta \rightarrow \sigma} \sum_{\tau: \tau \rightarrow \delta} \ip{e_\delta - e_\sigma^\Psi}{n_{\delta, \sigma}}\ip{d_\tau - d_\delta^\Psi}{n_{\tau, \delta}} V_\tau $$

$$+ \sum_{\delta: \delta \rightarrow \sigma} \bigl{(}  \ip{e_\delta - e_\sigma^\Psi}{n_{\delta, \sigma}} d_\delta^\Psi + \ip{d_\delta - d_\sigma^\Psi}{n_{\delta, \sigma}} e_\sigma^\Psi  \bigr{)} V_\delta + d_\sigma^\Psi e_\sigma^\Psi V_\sigma $$

$E \cdot (D \cdot D_\sigma)$ is of course the same expression with $e$ and $d$ switched.  Notice that the top row of the expression, by itself, is the basic cycle-level action.  Since the commutativity of that action was shown in \cite{Th}, it suffices to focus exclusively on the bottom half.  By the symmetry of the last term, it suffices to show that for $\delta \rightarrow \sigma$,

$$\ip{e_\delta - e_\sigma^\Psi}{n_{\sigma,\delta}} d_\delta^\Psi + \ip{d_\delta - d_\sigma^\Psi}{n_{\sigma,\delta}} e_\sigma^\Psi = \ip{d_\delta - d_\sigma^\Psi}{n_{\sigma,\delta}} e_\delta^\Psi + \ip{e_\delta - e_\sigma^\Psi}{n_{\sigma,\delta}} d_\sigma^\Psi$$

or equivalently,

$$\ip{d_\delta - d_\sigma^\Psi}{n_{\sigma,\delta}}(e_\delta^\Psi - e_\sigma^\Psi) = \ip{e_\delta - e_\sigma^\Psi}{n_{\sigma,\delta}}(d_\delta^\Psi - d_\sigma^\Psi)$$

Since $\delta \rightarrow \sigma$, $e_\delta^\Psi - e_\sigma^\Psi$ and $d_\delta^\Psi - d_\sigma^\Psi$ both lie in the one-dimensional space $\Psi(\delta)\cap M(\sigma)$.  The result follows immediately.


\end{proof}

\subsection{A Ring Structure on $Z(X)$ for Simplicial Fans}
\lbl{sec.simplicial}
Our goal in this section is to prove Theorem 2, which gives a ring stucture on the equivariant cycle group $Z_*^T(X)$ in the simplicial case.  This generalizes a similar result for the action on $Z_*(X)_\Q$ found in \cite{Th}.  We note that the (nonequivariant) cycle ring $Z_*(X)$ has a similar presentation as a quotient of a polynomial ring; however, this presentation does not appear in full generality in \cite{Th} or \cite{PT}.

 For this section, suppose that $\Sigma$ is simplicial,.  This implies that the group of divisors $\Div(X)$ is isomorphic to $Z_{n-1}(X)$, the group of cycles of codimension 1. Then clearly $\Div_T(X)$ and $Z_{n-1}^T(X)$ are also isomorphic as $\Lambda_\Q$-modules.  Under this isomorphism, the action of $\Div_T(X)$ on $Z_*^T(X)$ translates to a binary operation on $Z_*^T(X)$.  

\begin{proof}

 The first part of the proof is identical to the argument in \cite{Th}.  Since $\Sigma$ is simplicial, the cycles $\{D_i\}$ correspond to $\Q$-Cartier divisors.  Then, by the commutativity of the action, $Z_*^T(X)$ is a module over $\Lambda_\BQ[D_1,\dots, D_s]$.  In fact, $Z_*^T(X)$ is a cyclic module generated by $V_{\{0\}}$.  Indeed, let $\sigma \in \Sigma$, $\sigma = \rho_{i_1} + \rho_{i_2} + \dots + \rho_{i_k}$.  Then

$$D_{i_1} \cdot D_{i_2} \cdot \dots \cdot D_{i_k} \cdot V_{\{0\}} = rV_\sigma$$
where $r$ is a non-zero rational number, since the divisors $D_i$ intersect properly. (Note that if $X$ is non-singular, $r$ is always $1$.)    

Since $Z_*^T(X)$ is a cyclic module over $\Lambda_\BQ[D_1,\dots,D_s]$, $Z_*^T(X) \cong \Lambda_\BQ[D_1,\dots, D_s] / K$ for some ideal $K$.  It remains to show that $K = I + J^\Psi$.  We first prove $I \subset K$.  Let $D_{i_1} D_{i_2} \dots D_{i_k} \in I$.  Since the set $\{D_i\}$ intersects properly, this product must be a linear combination of $\{D_\tau: \rho_{i_1} + \rho_{i_2} + \dots + \rho_{i_k} \subset \tau\}$.  Since there are no such $\tau \in \Sigma$, the product is $0$.

Next we show $J^\Psi \subset K$.  Given any $m \in M_\sigma$, we consider the principal divisor $D$ with local equations $m_\tau = m$ for all $\tau \in \Sigma$.  Then

$$D_i D_j \dots D_k(\sum_{j=1}^s \ip{m^\Psi}{n_j} D_j - m^\Psi) = rV_\sigma \cdot ([D] - m^\Psi) = 0$$
by (\ref{eq.action}).

In order to prove $K \subset I + J^\Psi$, it must be shown that any polynomial $F \in \Lambda_\Q[D_1,\dots,D_s] / (I + J^\Psi)$ can be written as a linear combination of $\{V_\sigma : \sigma \in \Sigma \}$.  By the definition of $I$, this is equivalent to expressing $F$ as a square-free polynomial in $\Lambda[D_1,\dots,D_s]$.  Since square-free polynomials are closed under addition and multiplication by $\Lambda_\BQ$, we may assume without loss of generality that $F$ is a monomial term with coefficient $1$ such that $F \notin I$.
 
Renumbering the cones of $\Sigma$, we can assume that $F = D_1^{a_1} D_2^{a_2} \dots D_k^{a_k}$, where $\rho_1 + \rho_2 + \dots + \rho_k = \sigma \in \Sigma$, and $a_1 > 1$.  Our goal is to shift by an element of $J^\Psi$ so as to decrease the exponent $a_1$ without increasing any of the other non-zero exponents.  Then, by induction, we are done.  Since $\sigma$ is a simplicial cone, we can choose $m \in M_{\sigma, \Q}$ such that $\ip{m}{n_1} = 1$, and $\ip{m}{n_i} = 0$ for $2 \leq i \leq k$.  Then set

$$E = D_1^{a_1 - 1} D_2^{a_2} \dots D_k^{a_k} (\sum_{j=1}^s \ip{m^\Psi}{n_j} D_j - m^\Psi).$$

Clearly $E \in J^\Psi$, and $F - E$ is a sum of terms that all satisfy our requirement.  Thus $K = I + J^\Psi$, and the theorem is proven.
 
\end{proof}

It is easy to show that under the surjective homomorphism $Z_*^T(X) \rightarrow Z_*(X)_\Q$, $Z_*(X)_\Q \cong \Q[D_1,\dots,D_s] / I^\circ + (J^\psi)^\circ$, where $I^\circ$ and $(J^\psi)^\circ$ are the images of $I$ and $J^\Psi$, respectively.  This gives another characterization of the ring structure on $Z_*(X)_\Q$ of \cite{Th}.  In the case that $\Psi$ is induced from an inner product or a complete flag, there are more concrete descriptions of the ideal $J^\Psi$.   See \cite{Th} and [PT] for details in the basic case.

\subsection{Todd Class Expressions and SI-Interpolators}
\lbl{sub.todd}

In this section we prove Theorems 3, Theorem 4, and Corollary 1.4.  To this end, it will be useful to recall a few definitions and theorems regarding exponential sums and integrals over polyhedra.


Let $P$ be a polyhedron in $V := M_\Q$, with outer normal fan $\Sigma$.  Let $F$ be a face of $P$.  Let $\aff(F)$ be the affine span of $F$, or the smallest affine subspace of $V$ that contains $F$.  Let $\lin(F)$ be the linear subspace parallel to $\aff(F)$.  Since $P$ is rational, there is a natural lattice on $\lin(F)$.  By translation, there is a lattice measure $dm_F$ on $\aff(F)$.

Let $x$ be an interior point in $F$.  The tangent cone $\Tan(P,F) := \{v \in V | x + \epsilon v \in P \text{\ for\ some\ }x \in F^{\circ}, \epsilon > 0 \}$ is the cone of directions that one can go from any point $x$ in the interior of $F$ and stay in $P$.  The supporting cone $\Supp(P,F) := \Tan(P,F) + x$.  Both cones are independent of $x$.

The following theorem is a version of Lawrence's Theorem \cite{La}.

\begin{lawrence} 

Let $P$ be a polyhedron in $V := M_\Q$.  Then there exist meromorphic functions $S(P)$ and $I(P)$ on $V^*$ (i.e. elements of $\hat{\Lambda}_\Q$) with the following properties:

\begin{itemize}
\item
If $P$ contains a straight line, then $S(P)=I(P)=0$.
\item
$S$ (resp. $I$) is a {\em valuation} (resp. a {\em solid valuation}). That is, 
if the characteristic functions of a family of polyhedra satisfy a relation
$\sum_i r_i \chi(P_i)=0$, then the functions $S(P_i)$ satisfy
the relation 
$\sum_i r_i S(P_i)=0$ (resp. restrict the sum to those $P_i$ that do 
not lie in a proper affine subspace of $V$.)
\item
For every $v \in V$, we have 
\begin{equation}
I(v+P)=e^v I(P), 
\end{equation}
and 
\begin{equation}
S(v+P) = e^v S(P), \qquad v \in \La.
\end{equation}
\item
If $\xi \in V^*$ is such that $|e^{\la \xi,x \ra}|$ is integrable (resp. 
absolutely summable) over $P$, then

$$I(P)(\xi)=\int_{P}  e^{\la \xi,x \ra} dm_P(x), 
\qquad
S(P)(\xi)=\sum_{x \in P \cap \Lambda}  e^{\la \xi,x \ra}$$

where $dm_P$ denotes the relative Lebesgue measure on $\aff(P)$. The use of this measure prevents $I$ from being trivial on polyhedra that are not full-dimensional.
\end{itemize}

\end{lawrence}

Let $K$ be a non-singular cone in $V$ with primitive generators $v_1,\dots,v_k$.  Then a straightforward computation shows that 

\begin{equation}
I(K)= \frac{(-1)^k}{\prod_{i=1}^k v_i},
\qquad
S(K)= \prod_{i=1}^k  \frac{1}{1-e^{v_i}},
\end{equation}

We also recall the notion of $SI$-interpolator from \cite{GP}.

\begin{definition}

Let $V$ be a real vector space, and let $\calM(V^*)$ be the set of meromorphic functions on $V^*$.  Let $\calC$ be a set of cones in $V$.  An {\it SI-interpolator} is a map

$$\mu:\calC \rightarrow \calM(V^*)$$
such that given any rational polyhedron $P$ with $\Supp(P,F) \in \calC$ for all $F \subset P$,

$$S(P) = \sum_{F \subset P} \mu(\Supp(P,F))I(F)$$.

\end{definition}

%

We are now ready for the proofs of Theorems 3 and 4.

Let $\widehat {Z_*^T(X)}$ be the ring of power series in $Z_*^T(X)$.  Recall that for nonsingular toric varieties, the {\it equivariant Todd class} of $X$, denoted $\Td^T(X)$, has a product expression in $A_T^*(X)_\Q$. (See \cite{BV3}.)  Letting $\{B_i\}$ be the Bernoulli numbers for $1 \leq i < \infty$, then

$$\Td^T(X) = \prod_{i=1}^s (1 + \sum_{j=1}^\infty \frac{B_j}{j!} D_i^j) .$$

Using the induced ring structure on $Z_*^T(X)$, we can express $\Td^T(X)$ as a polynomial in $\{D_\sigma : \sigma \in \Sigma \}$ with coefficients in $\hat{\Lambda_\BQ}$.  The following lemma will show that the coefficient of $D_\sigma$ in this expression, which we will denote $r^\Psi(\sigma)$, is independent of the other cones in $\Sigma$.  Thus we can consider $\rsigma$ as a function of $\sigma$, independent of any fan that contains it.

\begin{lemma}

Let $F$ be any power series in $\Lambda_\Q[[D_1,\dots, D_s]]$, and let $\sigma \subset N$ be contained in two fans $\Sigma$ and $\Sigma'$, both generic with respect to $\Psi$, corresponding to toric varieties $X,X'$.  Let $F_\Sigma$, $F_{\Sigma'}$ be the images of $F$ in $Z_*(X)$, $Z_*(X')$ respectively.  Then the coefficient of $D_\sigma$ is the same in $F_\Sigma$ and $F_{\Sigma'}$.

\end{lemma}

\begin{proof}

By $\ref{eq.action}$, $D_\tau * D_\sigma$ is a sum of cycles corresponding to cones containing $\sigma$.  By commutativity, it is also a sum of cycles corresponding to cones containing $\tau$.  Thus the coefficient of $D_\sigma$ in $F_\Sigma$ or $F_{\Sigma'}$ only depends on the terms of $F$ corresponding to faces of $\sigma$.  Thus in calculating that coefficient, we may restrict $\Sigma$ or $\Sigma'$ to the fan consisting of $\sigma$ itself and all its faces.  Then the coefficient must be the same in both cases.
\end{proof}

 We will use the following lemma, a version of equivariant Riemann-Roch, which says that if one replaces cycles by the corresponding exponential integrals, then the Todd class is carried to the exponential sum.  (Note this is only true for nonsingular cones.) A version of this idea first appeared in \cite{KP}; see \cite{BR} for an elementary argument. 

We will use $\stilde$, $\itilde$, and $\mutilde$ to denote $S$, $I$, and $\mu$ with the substitution $\xi\mapsto -\xi$.

\begin{lemma}

Let $K=\Cone(m_1,\dots,m_n)$ be a full-dimensional nonsingular cone in $M$ and let $\sigma=\check{K}=\Cone(v_1,...,v_n)$, with $\{v_i\}$ denoting the basis of $N$ dual to $\{m_i\}$.   Let 
$$
R_{\sigma} = \frac{\Lambda_\Q[D_1,...D_n]}{J}, \ \ \  \ \ \ \ \ \ J=\biggl\langle \sum\ip{m}{v_i}D_i - m \ : \ m\in M \biggr\rangle .
$$
Let $\hat{R}_\sigma$ denote the completion of $R_\sigma$, and let $\hat{\Lambda_\Q}$ be the field of Laurent series over $M$.  Then there is a $\Lambda_\Q$-linear map
$\Phi: \hat{R}_\sigma \rightarrow \hat{\Lambda_\Q}$ such that
\begin{itemize}

\item For any subset $T\subset \{1,\dots,n\}$, with $K_T=\{ m\in K | \ip{m}{v_i} = 0, i\in T\}$, we have
$$
\Phi\bigl(\prod_{i\in T} D_i\bigr) = \itilde(K_T)
$$
\item and
$$
\Phi\biggl(\prod\frac{D_i}{1-e^{-D_i}}\biggr) = \stilde(K)
$$

\end{itemize}

\end{lemma}

\begin{proof} For any power series $\gamma$, 
map $\gamma(D_1, \dots, D_n)$ to $P^{-1}\gamma(m_1, \dots, m_n)$, where  $P=\prod_{i=1}^n m_i$ and extend by $\Lambda_\Q$-linearity.
We check that $J$ maps to 0:
$$
\Phi\bigl(\sum\ip{m}{v_i}D_i - m\bigr) = P^{-1}\bigl(\sum\ip{m}{v_i}m_i - m\bigr) = 0.
$$ 
Additionally,
$$
\Phi\bigl(\prod_{i\in T} D_i\bigr)=\prod_{i\notin T} m_i^{-1} = \itilde(K_T),
$$
the last equality using the fact that $K_T=\Cone(\{m_i | i\notin T \})$.
Finally,
$$
\Phi\biggl(\prod\frac{D_i}{1-e^{-D_i}}\biggr) = \prod \frac1{1-e^{-m_i}} = \stilde(K).
$$

\end{proof}

We now prove Theorem 3, Theorem 4 and Corollary 1.3 together.  We will first see that the interpolator equation of Corollary 1.3 holds for nonsingular cones.  This will enable us to argue inductively that the equation of Theorem 4 holds for nonsingular cones.  Once this relation between $r$ and $\mu$ is established in the nonsingular case, the fact that $r$ extends to a well-defined additive map on singular cones follows from the corresponding property of $\mu$. 

We now proceed to use Lemma 2.2 to prove the equation of Corollary 1.3 in the case that $K$ is a nonsingular $n$-dimensional cone.  We use the notation of the lemma. Note that by construction of the $\r(\sigma)$, it follows that we have the following equation in $Z_*^T(X)$:
$$
\prod\frac{D_i}{1-e^{-D_i}} = \sum_{\tau} r^\Psi (\tau) D_{\tau}.
$$
Applying the natural map $Z_*^T(X)\rightarrow R_\sigma$ followed by $\Phi$, we get
$$
\stilde(K) = \sum_{\tau\faceof\sigma} r^\Psi(\tau) \itilde(K_\tau),
$$
where $K_\tau=K\cap\tau^\bot$ is the face of $K$ dual to $\tau$.  For this face $F=K_\tau$, we have that the dual to the supporting cone $\Supp(K,F)$ is $\sigma_{K,F} = \tau$.  Thus we see that the equation of Corollary 1.3 holds in the case that $K$ is a full-dimensional nonsingular cone.

Our next lemma states that for cones $\sigma$ which are not necessarily of maximum dimension,  $r^\Psi(\sigma)$ can be computed by viewing $\sigma$ as a top-dimensional cone in the subspace $N_{\sigma}$ and then applying the inclusion $i^\Psi:M_{\sigma, \Q}\rightarrow M_\Q$.  We note that a complement map  on $N$  induces a complement map on all $\Psi$-generic sublattices, including $N_\sigma$.

\begin{lemma}
\lbl{lemma.lowerdim}

Let $\Psi$ be a complement map on $N$ and suppose that $\sigma$ is a $\Psi$-generic nonsingular cone in $N$. Let $\Psibar$ be the induced complement map on $N_\sigma$.  Let
 $\tilde{\sigma}=\sigma$, but considered as a cone in $N_{\sigma}$. Then
 $$
 r^\Psi(\sigma) = i^\Psi(r^{\Psibar}(\sigma_0)).
 $$

\end{lemma}

\begin{proof}
Let $X$ be the toric variety corresponding to the fan given by $\sigma$, and all of its faces, in $N$, and let $X_\sigma$ correspond similarly to $\sigma_0 \subset N_\sigma$.  Extend the inclusion map $i^\Psi: M_{\sigma, \Q} \rightarrow M_\Q$ to a map from $\Lambda_{\sigma, \Q} := \Q[M_\sigma]$ to $\Lambda_\Q := \Q[M]$.  This extends further to a natural map from $Z^T_*(X_\sigma)$ to $Z^T_*(X)$, using the characterization in Theorem \ref{thm.ringstructure}.  The well-definedness of this map hinges on the transitivity of the inclusions $i^\Psi$, which is part of the definition of a complement map.  This ensures that the ideal $J^\Psi_\sigma$ is sent to $J^\Psi$.

It is easy to see that this map preserves the product expression of the Todd class, and since the squarefree expression of the Todd class is unique in each ring, the map must preserve these expressions as well.  In particular, the coefficient of $V_{\tilde{\sigma}}$ is taken to the coefficient of $V_{\sigma}$.
\end{proof}

We now show that $r^\Psi(\sigma) = \mutilde^\Psi(\check{\sigma})$ holds for non-singular cones, which we argue by inducting both on $\dim M$ and $\dim \sigma$.  If $\sigma$ is full-dimensional, then we have seen that 
$$
\stilde(K) = \sum_{\tau\faceof\sigma} r^\Psi(\tau) \itilde(K_\tau),
$$
where $K=\check{\sigma}$, and $K_\tau= K \cap\tau^\bot$.  Since $\mu$ is an interpolator, we also have
$$
\stilde(K) = \sum_{\tau\faceof\sigma} \mutilde^\Psi(\check{\tau}) \itilde(K_\tau).
$$
For proper faces $\tau$ of $\sigma$, we may assume by induction, that $r^\Psi(\tau) = \mutilde^\Psi(\check{\tau})$, and it follows that $r^\Psi(\sigma) = \mutilde^\Psi(\check{\sigma})$ as well.

Now suppose that $\sigma$ is non-singular but not full-dimensional.  We apply Lemma \ref{lemma.lowerdim}.  With the notation of that lemma, we have
$$
r^\Psi(\sigma) = i^\Psi(r^{\Psibar}(\sigma_0)).
$$
Since the interpolator $\mu^\Psi$ is $\Psi$-hereditary (see \cite{GP}), we also have
$$
\mutilde^\Psi(\check{\sigma}) = i^\Psi(\mutilde^{\Psibar}(\check{\sigma_0})).
$$
But by induction,  $r^{\Psibar}(\sigma_0)=\mutilde^{\Psibar}(\check{\sigma_0})$.  Thus it follows that $r^\Psi(\sigma) = \mutilde^\Psi(\check{\sigma})$.

At this point, we have established that  $r^\Psi(\sigma) = \mutilde^\Psi(\check{\sigma})$ holds for nonsingular cones.  By the additive property of $\mutilde$, it follows that $\r$ may be extended to a well-defined  additive function on all cones (independent of subdivision), such that  $r^\Psi(\sigma) = \mutilde^\Psi(\check{\sigma})$ for all $\sigma$.  The analytic and interpolator properties of $\r$ follow from those of $\mu^\Psi$.  This completes the proof of Theorems 3 and 4 and Corollary 1.3.

Finally Corollary 1.4 follows directly from Theorem 4 together with the final assertion of Theorem 3.

\section{Examples}

In this section, we illustrate theorems of this paper by computing explicit formulas for $r^{\Psi}(\sigma)$ for nonsingular cones of dimension 2 or less with an arbitrary complement map $\Psi$.   We also match these formulas with those of \cite{GP}.

Our first two propositions concern cones of dimensions $0$ and $1$ in an arbitrary lattice.

\begin{proposition}
\lbl{prop.0d}

For any complement map $\Psi$ on $N$, we have $r^\Psi({0}) = 1$.

\end{proposition}

\begin{proof}
The constant term in the expansion of Equation (\ref{eq.toddprod}) is 1.  All other terms in this expansion are divisible by some $D_i$, which will remain true when the term is evaluated in $Z_*^T(X)$.  Thus these terms do not contribute to $r^\Psi({0})$.
\end{proof}

To state the formula for $r^\Psi({\sigma})$ where $\sigma$ is a cone of dimension greater than $0$, it will be useful to introduce  notation for the meromorphic function
$$
B(z) = \frac{1}{1-exp(z)} + \frac1z.
$$
Note that if we let $g(z) = \frac{z}{1-\exp(-z)}$ be the analytic function that defines the Todd class, then we have
\begin{equation}
\lbl{eq.gB}
g(z) = 1+ z B(-z)
\end{equation}

\begin{proposition}
\lbl{prop.1d}

Let $\Psi$ be a complement map on a lattice  $N$, and suppose $\sigma = \cone(\rho)$ is a one-dimensional cone generated by a primitive element $\rho\in N$.  Assume that  $\sigma$ is in the domain of $\Psi$ and let $c\in M$ be a generator of the one-dimensional space $\Psi(\sigma)$.  Then
$$
   r^\Psi(\sigma)= B\biggl(-\frac{c}{\ip{c}{\rho}}\biggr)
$$
\end{proposition}

\begin{proof}

Taking $\rho=\rho_1$, we expand Equation (\ref{eq.toddprod}) to find the coefficient of $D_1$, working modulo $D_j, j>1$.   According to Theorem 2, we find the relation:
$$
D_1 (\ip{c}{\rho} D_1 - c) = 0,
$$
from which it follows that $D_1^2=\frac{c}{\ip{c}{\rho}} D_1$,  and hence
$$
D_1^i = \biggl(\frac{c}{\ip{c}{\rho}}\biggr)^{i-1} D_1
$$
for all $i\geq 1$.  Hence, using Equation (\ref{eq.gB}), we find that
$$
g(D_1) =1+ B\biggl(-\frac{c}{\ip{c}{\rho}}\biggr) D_1.
$$
The proposition follows.
\end{proof}

\begin{proposition}
\lbl{prop.2d}

Let $\Psi$ be a complement map on a two-dimensional lattice $N$, and suppose $\sigma = \cone(\rho_1, \rho_2)$ is a two-dimensional non-singular cone  generated by a primitive elements $\rho_1, \rho_2\in N$.  Let $m_1, m_2$ be the dual basis of $M$, so that $\ip{m_i}{\rho_j}=\delta_{i,j}$  Assume that  $\sigma$ is in the domain of $\Psi$ and for $i=1,2$, let $c_i\in M$ be a generator of the one-dimensional space $\Psi(\rho_i)$.  Then
$$
   r^\Psi(\sigma)= B(-m_1) B(-m_2) -\frac1{m_2} \biggl(  B\biggl(-\frac{c_1}{\ip{c_1}{\rho_1}} \biggr)-   B(-m_1) \biggr)
   -\frac1{m_1} \biggl(  B\biggl(-\frac{c_2}{\ip{c_2}{\rho_2}} \biggr)-   B(-m_2) \biggr)
$$
\end{proposition}

In the equation above, note that in spite of the appearance of fractions $\frac1{m_i}$ on the right hand side, the expression given is actually a power series.  Indeed, using the fact that $B(z) - B(w)$ is divisible by $z-w$, one easily sees that these denominators cancel.

\begin{proof}
We wish find the coefficient of $D_1D_2$ in $g(D_1) g(D_2)$ working in the ring $Z_*^T(X)$ modulo $D_j, j>2$. We have the relations
$$
D_1(\ip{c_1}{\rho_1} D_1 + \ip{c_1}{\rho_2} D_2 - c_1) = 0, \ \ \ \ \ D_2(\ip{c_2}{\rho_1} D_1 + \ip{c_2}{\rho_2} D_2 -c_2) = 0,
$$
and
$$
   D_1D_2(D_1- m_1) = 0  \ \ \ \ \ \ \ \  D_1D_2(D_2- m_2) = 0.
$$

These imply, for $i,j>0$, that
$$
D_1^i D_2^j = m_1^{i-1} m_2^{j-1}  D_1 D_2
$$

We also get
$$
D_1^2 = -\frac{\ip{c_1}{\rho_2}}{\ip{c_1}{\rho_1}} D_1 D_2 + L_1 D_1 \ \ \ \ \ 
$$
where $L_j$ is defined as $\frac1{\ip{c_j}{\rho_j}} c_j$.   By induction, it follows that for $i\geq 2$,
$$
D_1^i = -\frac{\ip{c_1}{\rho_2}}{\ip{c_1}{\rho_1}} ( m_1^{i-2} + m_1^{i-3} L_1+ \cdots + L_1^{i-2} ) D_1 D_2 + L_1^{i-1} D_1 \ \ \ \ \ 
$$
Now using $ m_1^{i-2} + m_1^{i-3} L_1+ \cdots + L_1^{i-2}= \frac{L_1^{i-1}-m_1^{i-1}}{L_1-m_1}$ and $L_1-m_1=\frac{\ip{c_1}{\rho_2}}{\ip{c_1}{\rho_1}} m_2$, 
$$
D_1^i = - \frac{1}{m_2} (  L_1^{i-1}-m_1^{i-1} ) D_1 D_2 +  L_1^{i-1} D_1,
$$
which holds also for $i=1$.
Hence
$$
D_1 B(-D_1) =- \frac{1}{m_2} (  B(-L_1)-B(-m_1)) D_1 D_2 +  B(-L_1) D_1,
$$
with a similar formula for $D_2 B(-D_2)$. Thus
the coefficient of $D_1 D_2$ in $ g(D_1) g(D_2) = (1 + D_1 B(-D_1)) ((1 + D_2 B(-D_2))$ is given by
$$
 -\frac{1}{m_2}(B(-L_1)-B(-m_1)) -  \frac{1}{m_1}(B(-L_2)-B(-m_2)) + B(-m_1)B(-m_2),
 $$
 as desired.

\end{proof}

We note that Propositions \ref{prop.1d} and \ref{prop.2d} allow one to check directly, for nonsingular cones $\sigma$ of dimension at most 2, the agreement of this paper's $r^\Psi$ with the $\mu^\Psi$ of \cite{GP}, as asserted in Theorem  \ref{thm.gpagree}.   Indeed, with the substitution of $\xi$ for $-\xi$, one finds exact agreement of Proposition $\ref{prop.1d}$ above with \cite{GP}, Propostion 5.3, in both the complete flag and inner product cases.  Similarly, in dimension 2, one sees agreement between Proposition \ref{prop.2d} and \cite{GP}, Proposition 5.4.

 \begin{example}

We consider a nonsimplicial fan $\Sigma$  in $N=\BZ^3$ as follows.  Let  $\rho_1=(1,0,0),\rho_2=(0,1,0),\rho_3=(0,0,1),\rho_4=(1,1,-1)$, and $\rho_5=(-1,-1,0)$.  The cone generated by $\rho_i,\rho_j, \dots, \rho_k$ will be denoted  $\sigma_{i,j,\dots,k}$, with $\sigma_\emptyset = \{0\}.$  Consider the fan $\Sigma$ with rays $\rho_i, i=1,\dots, 5$, whose maximal cones are $\sigma_{1234}, \sigma_{135}, \sigma_{145}, \sigma_{235}$, and $\sigma_{245}$.  Then $\Sigma$ is complete fan with  eight two-dimensional cones $\sigma_{13}$, $\sigma_{14}$, $\sigma_{15}$, $\sigma_{23}$, $\sigma_{24}$, $\sigma_{25}$, $\sigma_{35}$, $\sigma_{45}$.   Let $X=X_{\Sigma}$, and for each cone $\sigma_{i,j,\dots,k}\in\Sigma$, let $V_{i,j,\dots,k}$ denote the corresponding cycle on $X$.  Note that affine toric variety corresponding to $\sigma_{1234}$ is the affine cone on $\P^1\times\P^1$.

Let $\Psi$ denote the complement map obtained from the standard inner product on $\BQ^3$. We will illustrate the action of $\Div_T(X)$ on $Z^T_*(X)$ given in Theorem 1.  First note that for $D\in \Div(X)$, the local equations $\{d_\sigma\}_{\sigma\in\Sigma}$ specify a continuous piecewise $\BQ$-linear function on $\Sigma$ that is determined by its values on $\rho_i, i=1,\dots, 5$.  Denote these values by $\alpha_i, i=1,\dots, 5$.  The $\alpha_i$ are any rational numbers that satisfy $\alpha_1+\alpha_2=\alpha_3+\alpha_4.$  For such $D$, we compute the $D\cdot V_{\sigma}$ for various $\sigma\in\Sigma$.

First, let $\sigma=\sigma_{\emptyset}=\{0\}$.   Then $d_\sigma^\Psi= 0$, and the formula in Theorem 1 gives
$$
D\cdot V_{\emptyset} = \alpha_1 V_1 + \alpha_2 V_2 +\alpha_3 V_3 +\alpha_4 V_4 +\alpha_5 V_5.
$$
This is simply the Weil divisor determined by the Cartier divisor $D$.

Next, let $\sigma=\sigma_1=\rho_1$.   Then using $\{m_1, m_2, m_3\}$ to denote the standard basis of $M=\BZ_3^*$, we have $d_\sigma^\Psi= \alpha_1 m_1$, and the formula in Theorem 1 gives
$$
D\cdot V_1 = \alpha_3 V_{13} + (\alpha_4-\alpha_1) V_{14} + (\alpha_1+\alpha_5) V_{15} + \alpha_1 m_1 V_1.
$$
For $i\ne 1$, $D\cdot V_i$ may be computed similarly.

Turning to a $2$-dimensional cone, say $\sigma=\sigma_{13}$, we find $d_\sigma^\Psi= \alpha_1 m_1 + \alpha_3 m_3$, and 
$$
D\cdot V_{13} = \alpha_2 V_{1234} + (\alpha_5 + \alpha_1) V_{135} + (\alpha_1 m_1 + \alpha_3 m_3) V_{13}.
$$
 
Finally, take $\sigma=\sigma_{1234}.$  Then $d_\sigma^\Psi= \alpha_1 m_1 + \alpha_2 m_2 + \alpha_3 m_3$,
 and 
 $$
 D\cdot V_{1234} = ( \alpha_1 m_1 + \alpha_2 m_2 + \alpha_3 m_3)  V_{1234}.
 $$

 \end{example}

\begin{example}

As a final example, we consider the two-dimensional nonsingular triangle  in $M=\Z^2$ with vertices $v_0=(0,0), v_1=(1,0), v_2=(0,1).$ The corresponding inner normal fan $\Sigma$ in $N=\Z^2$ has rays generated by $\rho_0=(-1,-1), \rho_1=(0,1), \rho_2=(1,0)$. The two-dimensional cones of this fan are
$\sigma_0=\Cone((0,1),(1,0))$, $\sigma_1=\Cone((-1,-1),(0,1))$, $\sigma_2=\Cone((0,1),(-1,-1))$.

Let $\Psi$ be the complement map induced by the standard inner product on $\Z^2$.  We take  $\{x=(1,0), y=(0,1)\}$, to be the standard basis of $M$, so that $\Lambda=\Z[x,y]$.  Then one calculates  according to Theorem \ref{thm.ringstructure} that the ring structure on the equivariant cycle groups is given by

$$Z^T_*(X) \cong \frac{\Lambda_\BQ[D_0,D_1,D_2]}{I + J^\Psi}$$

where $I = \langle D_0D_1D_2 \rangle $, and

$$J^\Psi = \biggl\langle   D_1(D_1-D_0-x), \  D_2(D_2-D_0-y), \  D_0(2D_0-D_1-D_2+x+y)
 \biggr\rangle.
$$

To compute the equivariant Todd class of $X_\Sigma$, one multiplies out the expression $\Pi_{i=0}^2 \frac{D_i} {1-\exp(-D_i)}$ in the completion of the above ring. Either working directly with the relations above, or using Proposition \ref{prop.2d},  one finds that the coefficient $r(\sigma_0)$ of $D_1D_2$ is given by
$$
r(\sigma_0)=B(-x)B(-y).
$$
Likewise, the coefficient  $r(\sigma_1)$ of $D_0D_2$ is 
$$
r(\sigma_1) = B(x-y)B(x)- \frac1{x}\biggl[    B(-y ) -B(x-y)  \biggr]  - \frac1{y-x}\biggl[    B(\frac12 (x+y) ) -B(x)  \biggr],
$$
and the coefficient of $D_0D_1$ is
$$
r(\sigma_2) = B(y)B(y-x)- \frac1{x-y}\biggl[    B(\frac12 (x+y) ) -B(y)  \biggr] + \frac1{y}\biggl[    B(-x ) -B(y-x)  \biggr].
$$

One notes that the $r(\sigma_i)$ may be expanded in power series about the origin (shown here to order 2):
\begin{align*}
r(\sigma_0) &= \frac14 + \frac1{24} x + \frac1{24} y + \frac1{144} xy +\cdots
\\  r(\sigma_1) &= \frac38 - \frac1{12} x + \frac1{24} y + \frac5{1152} x^2 - \frac1{288} xy - \frac5{1152} y^2 + \cdots
\\ r(\sigma_2) &=  \frac38  +  \frac1{24} x - \frac1{12} y - \frac5{1152} x^2 - \frac1{288} xy + \frac5{1152} y^2 + \cdots
\end{align*}
and one recovers the constant terms from the local lattice point formula of \cite{PT}, namely $\mu_0(\sigma_0)=\frac14,\mu_0(\sigma_1)=\frac38,\mu_0(\sigma_2)=\frac38$. (cf. \cite{GP}, Example 5.7).

Finally, one can verify by direct computation that $r(\sigma_0)+r(\sigma_1)+r(\sigma_2)=1$, in agreement with Theorem \ref{thm.rsigma}.  One can also use the above expressions for the $r(\sigma_i)$ to verify the $SI$-interpolator property of Corollary \ref{thm.rinterpolates}:
$$
 \sum_F r^{\Psi}(\sigma_{P,F})(-\xi)\cdot I(F)(\xi)=S(P)(\xi)=1+e^x+e^y.
$$

\end{example}

\ifx\undefined\bysame
        \newcommand{\bysame}{\leavevmode\hbox
to3em{\hrulefill}\,}
\fi

\end{document}